\documentclass[11pt,leqno]{amsart}

\usepackage{newtxtext,times,txfonts,dsfont,eucal,graphicx}

\usepackage[text={6.5in,8.9in},centering]{geometry}

\usepackage[dvipsnames]{xcolor}   
\usepackage{xparse}
\usepackage{xr-hyper}
\usepackage[linktocpage=true,colorlinks=true,hyperindex,citecolor=RoyalBlue,linkcolor=RoyalBlue]{hyperref}

\newcommand{\inv}{^{\raisebox{.2ex}{$\scriptscriptstyle-1$}}}   
\newcommand{\spec}{\mathcal{P}(L)}
\newcommand{\specz}{\mathcal{P}_z(L)}
\newcommand{\mx}{\mathcal{M}(L)}
\newcommand{\jr}{\mathsf{j}_L}
\newcommand{\z}{\mathcal{Z}(L)}
\DeclareMathOperator{\irr}{\mathcal{I}_z}
\DeclareMathOperator{\irs}{\mathcal{I}_z^+}
\newcommand{\an}{\mathsf{a}(b)}
\DeclareMathOperator{\cz}{\mathrm{c}\;\!\!\ell}

\newtheorem{theorem}[equation]{Theorem}
\newtheorem{proposition}[equation]{Proposition}
\newtheorem{lemma}[equation]{Lemma}
\newtheorem{corollary}[equation]{Corollary}
\theoremstyle{definition}
\newtheorem{cremark}[equation]{Concluding remarks}

\newtheorem{remark}[equation]{Remark}

\numberwithin{equation}{section}

\begin{document}

\author{Themba Dube}

\address{[1] Department of Mathematical Sciences, University of South Africa.
[2] National Institute for Theoretical and Computational Sciences (NITheCS).
}
\email{dubeta@unisa.ac.za}

\author{Amartya Goswami  
}

\address{[1] Department of Mathematics and Applied Mathematics, University of Johannesburg.
[2] National Institute for Theoretical and Computational Sciences (NITheCS).
}

\email{agoswami@uj.ac.za}

\title{On  $z$-elements of multiplicative lattices}

%quantales

%None of the above, but in this section

%Complete lattices, completions

%

\maketitle

\begin{abstract}
The aim of this paper is to investigate further properties of $z$-elements in multiplicative lattices. We utilize $z$-closure operators to extend several properties of $z$-ideals to $z$-elements and introduce various distinguished subclasses of $z$-elements, such as $z$-prime, $z$-semiprime, $z$-primary, $z$-irreducible, and $z$-strongly irreducible elements, and study their properties. We provide a characterization of multiplicative lattices where $z$-elements are closed under finite products and  a representation of $z$-elements in terms of $z$-irreducible elements in  $z$-Noetherian multiplicative lattices.
\end{abstract}
  
\section{Introduction}\label{intro}

The purpose of multiplicative lattices, introduced in \cite{WD39} (also see \cite{Dil62}), is to provide an abstract framework for studying the ideal theory of commutative rings with identity. Since then, extensive research has been conducted to extend various types of ideals to corresponding types of elements in multiplicative lattices. It is unsurprising, then, that $z$-ideals of commutative rings are part of this endeavor as well. Although the notion of $z$-ideals (of rings of continuous functions) was introduced by Kohls \cite{Koh57} in the late 1950s, leading to numerous studies (see, e.g., \cite{GJ60, AKA99, AM07, AP20, Mar72, HMW03, Don80, Mas80, Man68, Mas73, HP80-I, HP80-II, Pag81}) and their generalizations (such as \cite{AM13, MB20, DI16, ABN20, BM20, MB22, AAT13, MJ20, XZ09, JK19, Dub16, Dub18, VMS19, Moh14}), to the best of the our knowledge, the extension of $z$-ideals to $z$-elements of multiplicative lattices was only recently undertaken in \cite{MC19}.

Our focus in this paper is twofold. Firstly, we explore further properties of $z$-ideals and extend them to $z$-elements in multiplicative lattices. Secondly, we introduce several subclasses of $z$-elements and study their properties. In both cases, we systematically use certain closure operators. In the context of $z$-ideals of commutative rings, for any ideal $I$, there exists a smallest $z$-ideal $I_z$ (see \cite{Mas73}) containing $I$. Similarly, for an element $i$ in a multiplicative lattice, this smallest $z$-element is denoted by $i_z$ in \cite{MC19}. We refer to these smallest $z$-elements as `$z$-closure operators' because, like any closure operator, they provide an alternative definition for $z$-elements (see Proposition \ref{lclk}(\ref{altd})). Moreover, $z$-closure operators serve the purpose of examining  subclasses of $z$-elements introduced here.

The motivation for this approach arises from a result in \cite{JRT22}, where it was shown that:
\begin{equation}\label{kppk}
k\text{-prime ideal}\Leftrightarrow k\text{-ideal} \;+\;\text{prime ideal,}
\end{equation}
where a $k$-ideal is defined in the sense of \cite{Hen58}, and a $k$-prime ideal is obtained by restricting the definition of a prime ideal to $k$-ideals (see \cite[Definition 3.4(1)]{JRT22}).

Motivated by (\ref{kppk}), we first introduce some distinguished subclasses of $z$-elements of a multiplicative lattice. Using our $z$-closure operators, we establish the following equivalence formulations:
\begin{equation}\label{kxkx}
z \text{-}\star\!\text{ element}\Leftrightarrow z\text{-element}\; + \;\star\; \text{element},
\end{equation}
where `$\star$' stands for irreducible,  strongly irreducible, and, under an additional assumption, for $z$-prime, $z$-semiprime, and $z$-primary elements.

For the sake of completeness, unless it is obvious, every result presented in this paper will be accompanied by either a proof or a reference where the proof can be found. Additionally, when we extend a result to $z$-elements, we shall indicate a source for the  ring-theoretic counterpart of it.

We now briefly describe the content of the paper. In Section \ref{prlm}, we gather the required background on multiplicative lattices, while in Section \ref{zid}, we study some properties of $z$-elements, which includes  a characterization of multiplicative lattices in which $z$-elements are closed under finite products (Theorem \ref{prdz}). In Section \ref{zco}, we study properties of $z$-closure operators and extend several results on $z$-ideals to $z$-elements. We give a characterization of $pz$-multiplicative lattices (Proposition \ref{cpzl}). We also discuss some results (in Section \ref{sonu}) on $z$-closure operators viewed as nuclei of various types. The purpose of Section \ref{sczi} is to introduce a number of subclasses of $z$-elements, namely, $z$-irreducible, $z$-meet irreducible, $z$-prime, $z$-semiprime, and $z$-primary, and prove the equivalence formulation (\ref{kxkx}) for them. We show a representation of any $z$-element in terms of $z$-irreducible elements in a $z$-Noetherian multiplicative lattice (Proposition \ref{nssi}). We prove a couple of results on minimal $z$-prime elements (Proposition \ref{kmkp} and  Proposition \ref{wnknp}). We conclude the paper with a discussion of future work.

\section{Background} \label{prlm}
 
In this section, we shall review key definitions and results on multiplicative lattices that are essential for this paper. For a detailed study of multiplicative lattices, readers may consult \cite{And74, And76, Dil62, Dil36, War37, WD39}. According to \cite{Dil62},  a \emph{multiplicative lattice}\footnote{Various definitions exist in the literature. For instance, according to \cite{Kei72} (also see \cite{FFJ22}), a \emph{multiplicative lattice} is defined as a complete lattice $(L, \leqslant , 0, 1)$ equipped with a multiplication operation on $L$ such that $ab\leqslant a\wedge b$ holds for all $a$, $b\in L$. Another definition can be found in \cite{Ros87}.} is a complete lattice $(L, \leqslant, 0, 1)$ endowed with an associative, commutative multiplication (denoted by $\cdot$), which distributes over arbitrary joins and has $1$ as multiplicative identity. Note that  in the sense of \cite{Mul86},  a multiplicative lattice is a commutative, unital quantale.
For brevity, we shall write $xy$ for $x\cdot y$ and  $x^n$ for $x\cdot   \cdots \cdot   x$ (repeated $n$ times). We say $x$ is \emph{strictly below} $y$, denoted by $x<y$, whenever $x\leqslant y$ and $x\neq y$. 

For our later purposes, in the following lemma, we gather some elementary properties of multiplication of multiplicative lattices.

\begin{lemma}\label{bip}
In a multiplicative lattice $L$, the following hold.
\begin{enumerate}
	
\item\label{pxyx} $xy\leqslant x$, for all $x$, $y\in L.$
	
\item \label{mul}
$x  y\leqslant x\wedge y$, for all $x, y\in L$.
		
\item $x  0=0$, for all $x\in L$.
		
\item\label{mon} If $x\leqslant y$, then $x  z\leqslant y  z,$\, for all $x$, $y$, and $z\in L$.
		
\item\label{monj} If $x\leqslant y$ and $u\leqslant v$, then $x  u\leqslant y v$, for all $x,$ $y,$ $u,$ $v\in L$.
\end{enumerate}
\end{lemma} 

\begin{proof}
For (1), 	we notice that \[x=x  1=x  (y\vee 1)=(x y)\vee x,\] and that implies $x  y\leqslant x$. The claim (2)  Follows from (\ref{pxyx}). For	
(3), we  apply (\ref{mul}) and obtain $x  0\leqslant x\wedge 0=0.$ Hence $x  0=0.$ Since $x\leqslant y$, we have \[y  z=(x\vee y)  z=(x  z)\vee (y  z),\] which implies that $x  z\leqslant y  z$, and this proves (4). Finally, to have
(5), we observe that $x\leqslant y$ and $u\leqslant v$. Applying (\ref{mon}), we get $x  u\leqslant y  u$ and $y  u\leqslant y  v$, and hence $xu\leqslant yv$. 
\end{proof} 

We shall now recall a few definitions. An element $x$ of a multiplicative lattice $L$ is called \emph{proper} if $x\neq 1$. 
A proper element $p$ of $L$ is called \emph{prime} if $x y\leqslant p$ implies that $x\leqslant p$ or $y\leqslant p$ for all $x$, $y\in L$. We shall denote by $\spec$ the set of prime elements of $L$.   A proper element $m$ of $L$ is said to be \emph{maximal}, if $m \leqslant x$ and $x< 1$ implies that $m = x$ for all $x\in L$.   We shall denote by $\mx$ the set of all maximal elements of $L$. An element $c$ of $L$ is called \emph{compact} if for any $\{x_{\lambda}\}_{\lambda \in \Lambda}\subseteq L$ and $c\leqslant \bigvee_{\lambda \in \Lambda} x_{\lambda}$ implies that $c\leqslant \bigvee_{i=1}^n x_{\lambda_i}$, for some $n\in \mathds{N}^+$, and  $L$ is said to be \emph{compactly generated} if every element of
$L$ is the join of compact elements of $L$.  

The following well-known results (see, e.g., \cite[Lemma 2.2]{JS86} for (\ref{exml})) are going to play an important role in this paper. Although the proofs are easy, we include them here for the sake of completeness.

\begin{lemma}\label{flm}
For a multiplicative lattice $L$, the following hold.
\begin{enumerate}
\item\label{mxip} If $m\in \mx$, then $m\in \spec$.
	
\item\label{exml} If the top element $1$ of $L$ is compact, then for every proper element $a$ of $L$, there exists an $m\in \mx$ such that $a\leqslant m$.
\end{enumerate}
\end{lemma}

\begin{proof}
(1). Let $m\in \mx$ and let  $x$, $y\in L$ such that $xy\leqslant m$. If $x\nleqslant m$, then maximality of $m$ yields $m\vee x=1$. Therefore, \[y=y1=y(m\vee x)=ym\vee yx\leqslant m\vee m=m,\]  showing $m\in \spec$.

(2). We consider  a  subset
\[S =\{ x\in L\mid a\leqslant x< 1\}\]
of $L$.
Then $S\neq \emptyset$ as $a\in S.$ We consider $S$ as a poset, with the partial order inherited from $L$. Let $C\subseteq S$
be a chain. Take $c=\bigvee C$. We shall show that $c$ is an upper bound for $C$ and $c\in C$. That $c$
is indeed an upper bound for $C$ is immediate. To show that $c\in C$, we need only to show that
$c\neq 1$. Suppose, by way of contradiction, that $c=1$. Since $1$ is compact, there are finitely
many  $c_1, \ldots, c_n\in C$ such that \[c_1\vee \cdots \vee c_n=1.\] Since $C$ is a chain, there is an
index $k\in \{1, \ldots, n\}$  such that $c_k=1$; which is false because $1\notin C$. So, by Zorn’s Lemma, $S$
has a maximal element, say $m$. To be done, we must show that $m\in \mx$.
We consider any $n\in L$ with $m\leqslant n< 1$. Then $n\in C$, whence $m=n$, because $m$ is a maximal
element in $C$.
\end{proof}

\section{$z$-elements}
\label{zid}

First, let us list the assumptions that we will impose on our multiplicative lattices. We shall work with  compactly generated multiplicative lattices in which 1 is compact and every finite
product of compact elements is a compact element.

From Lemma \ref{flm}(\ref{exml}) it follows  that $\mx\neq \emptyset$. Therefore, the notion of  $z$-ideals of  (commutative) rings can be extended to  multiplicative lattices.
For an $a\in L$, let us consider \[\mathcal{M}_a^L=\left\{ m\in \mx \mid a\leqslant m\right\}\quad \text{and} \quad \mathsf{m}_a= {\bigwedge}\mathcal{M}_a^L.\]
An element $x$ of  $L$ is called a \emph{z-element} if for all $a$, $b\in L$ with $\mathcal{M}^L_a=\mathcal{M}^L_b$ and $b\leqslant x $ implies that $a\leqslant x$. We shall denote by $\z$ the set of all $z$-elements of  $L$.  When there is no ambiguity concerning the underlying multiplicative lattice, we shall write $\mathcal{M}_a$ for $\mathcal{M}_a^L$.

\begin{remark}
It can be shown (see \cite[Lemma 2.10]{MC19}) that  an element $x$ is a $z$-element if and only if $\mathcal{M}^L_a\supseteq \mathcal{M}^L_b$ and $b\leqslant x $ implies that $a\leqslant x$. 	
\end{remark}

It turns out that $z$-elements abound in multiplicative lattices. Here, we present some examples of these elements, which we will verify later. Moreover, we will explain related terminologies as they arise.

\begin{enumerate}
\item [$\bullet$]
The top element $1$ of $L$ is a $z$-element (Proposition \ref{lclk}(\ref{ckr})).

\item [$\bullet$]
Every maximal element of $L$ is a $z$-element (Lemma \ref{epzi}(\ref{meze})).

\item [$\bullet$]
The Jacobson radical $\jr$ is a $z$-element (Lemma \ref{epzi}(\ref{jrze})).

\item [$\bullet$]
If $a$ is a $z$-element, so is the residual $(a:b)$ (Proposition \ref{icj}).

\item [$\bullet$]
If $L$ is semisimple, then annihilator $\an$ of $x$ is a $z$-element (Corollary \ref{anze}).
\end{enumerate}

In the next lemma, we compile some of the elementary properties of $z$-elements of multiplicative lattices that will be utilized in the sequel. Let us recall the following definitions that are necessary to state the lemma. The \emph{Jacobson radical} of a multiplicative lattice $L$ is defined by \[\jr=\bigwedge\left\{m\mid m\in \mx \right\},\]  and a prime element $p$ of  $L$ is said to be \emph{minimal} if there is no 
$q\in \spec$  such that $q < p.$

\begin{lemma}\label{epzi}
In a multiplicative lattice $L$, $z$-elements have the following properties.
\begin{enumerate}
		
\item\label{ijam} If $\{x_{i}\}_{i\in I}\subseteq \z$, then  $\bigwedge_{i\in I} x_i\in \z$.

\item\label{alze} $x\in \z$  if and only if $\mathsf{m}_x= x$. 		

\item\label{meze} If $m\in \mx$, then $m\in \z$.
		
\item If $L$ has a unique maximal element m, then any element $a < m$ is not a $z$-element.

\item\label{jrze}  $\jr\in \z$. 
		
\item\label{mpz} If $p$ is minimal in the class of prime elements above a $z$-element $x$, then $p\in \z$.

\item \label{ziz} If $\jr=0$, then $0\in \z$.
\end{enumerate}
\end{lemma}

\begin{proof}
 
(1). Suppose $a$, $b\in L$,  $\mathcal{M}_a\supseteq \mathcal{M}_b$, and $b\leqslant \bigwedge_{i\in I}x_i$. Since $\bigwedge_{i\in I}x_i\leqslant x_i$ and for all $i\in I$, $x_i\in \z$,  we have $a\leqslant x_i$ for all $i\in I$, and hence, $a\leqslant \bigwedge_{i\in I}x_i,$ implying that $\bigwedge_{i\in I}x_i\in \z.$

(2). Suppose  $\mathsf{m}_x =x$ for some $x\in L$. Let $\mathcal{M}_a\supseteq\mathcal{M}_b$ and $b\leqslant x, $ for some $a$, $b\in L$. Since $\mathcal{M}_a\supseteq \mathcal{M}_b$, we have $a\leqslant m$, for all $m\in \mathcal{M}_b$.  Since $b\leqslant x=\bigwedge\left\{m\in \mathcal{M}(L)\mid x\leqslant m\right\}$, this implies that  \[a\leqslant \mathsf{m}_b\leqslant \mathsf{m}_x= x,\] and hence $x\in \z$. For the converse, suppose  $x\in \z.$ Then from \cite[Lemma 2.10]{MC19} it follows that $\mathsf{m}_x\leqslant x$, whereas $x\leqslant\mathsf{m}_x$ follows from the definition of $\mathsf{m}_x$.

(3)--(6). See Lemmas 2.2, 2.3, 3.9, and 3.4 respectively in  \cite{MC19}.

(7). Follows from (\ref{ijam}) and (\ref{jrze}). 
\end{proof}

\begin{remark}
Among the results in Lemma \ref{epzi}, (\ref{mpz}) extends \cite[Theorem 1.1]{Mas73}, whereas (\ref{alze}) gives an alternative definition of a $z$-element. We shall freely use any one of the three equivalent definitions of a $z$-element throughout the remainder of the paper.
\end{remark}

We shall now construct  examples of $z$-elements of multiplicative lattices.  Suppose that $L$ is a multiplicative lattice. For any $a,$ $b\in L$,  the \emph{residual} of $a$ by $b$ is defined by \[(a:b)=\bigvee\left\{ l\in L\mid lb\leqslant a\right\}.\] 
The following proposition serves as an extension of \cite[Proposition 1.3]{Mas73}.

\begin{proposition}\label{icj}
Suppose $L$ is a multiplicative lattice. If $a\in \z$ and $b\in L$, then $(a:b)\in \z$.
\end{proposition}

\begin{proof}
Let $\mathcal{M}_x\supseteq \mathcal{M}_y$ and $y\leqslant (a : b).$ This implies that \[yb\leqslant (a:b)b\leqslant a.\] Since $\mathcal{M}_x\supseteq \mathcal{M}_y$, applying \cite[Lemma 2.6(2)]{MC19}, we have $\mathcal{M}_{xb}\supseteq \mathcal{M}_{yb}$. Since $a\in \z$, we have $xb\leqslant a$. Hence $x\leqslant (a : b).$
\end{proof}

\begin{corollary}
Let $L$ be a multiplicative lattice. If   $b$, $\{b_{i}\}_{i \in I}$, $c\in L$  and   $a$, $\{a_{j}\}_{j \in J}\in \z$, then    $((a : b) : c),$ $(a : b c),$ $((a : c) : b),$ $ (\bigwedge_{j\in J} a_{j} : b),$ $ \bigwedge_{j\in J}(a_{j}: b),$ $ (a : \bigvee_{i\in I}b_{i}),$ and $ \bigwedge_{i\in I}(a : b_{i})$ are $z$-elements of $L$.
\end{corollary}

When $L$ is semisimple (that is, $\jr=0$), an annihilator is also an example of a $z$-element. Recall that  the \emph{annihilator} of an element $b$ of $L$ is defined by \[\an=(0:b)=\bigvee\left\{ l\in L\mid lb=0\right\}.\]

\begin{corollary}\label{anze}
If $L$ is a semisimple multiplicative lattice, then $\an\in \z$, for all $b\in L$.
\end{corollary}

\begin{proof}
The claim follows from Lemma \ref{epzi}(\ref{ziz}) and Proposition \ref{icj}.
\end{proof}

From Lemma \ref{flm}(\ref{mxip}), we know that every maximal element of a multiplicative lattice is prime. The following proposition gives a sufficient condition for the converse to hold, and moreover, it  extends \cite[Lemma 2.13]{AM22}. To state the result, we first require a few definitions.
Suppose that $L$ is a multiplicative lattice. We recall that if $L$ has finitely many maximal elements, then $L$ is called \emph{quasi-local}.  Following the ring-theoretic version (see \cite[Definition 2.10]{AM22}), we say $L$ is a \emph{$pz$-multiplicative lattice} if every prime element of $L$ is a $z$-element.

\begin{proposition}
If $L$ is a quasi-local $pz$-multiplicative lattice, then every prime element is maximal.
\end{proposition}

\begin{proof}
Let $p\in \spec$. Since $L$ is a $pz$-multiplicative lattice, $p\in \z$. Therefore, \[\jr=\bigwedge_{i=1}^n \left\{m_i\mid m_i\in\mx\right\}\leqslant\mathsf{m}_p\leqslant p,\]
where the last inequality follows from Lemma \ref{epzi}(\ref{alze}).
Since $p\in \spec$, there is an index $j\in \{1, \ldots, n\}$ such that $m_j\leqslant p$. Since $m_j\in \mx, $ we must have $m_j=p$. Hence $p\in \mx.$
\end{proof}

Our next theorem shows a relation between prime elements and $z$-elements in semisimple multiplicative lattices. Moreover, it extends \cite[Theorem 1.5]{Mas73}.

\begin{theorem}
Suppose $L$ is a semisimple multiplicative lattice. If $p\in \spec$, then either $p\in \z$ or the maximal
$z$-elements that are below $p$ are prime $z$-elements.
\end{theorem}

\begin{proof}
Let us suppose \[S=\left\{x\in \z \mid x\leqslant p\right\}.\] Since $L$ is semisimple, by Lemma \ref{lclk}(\ref{ssz}), $0\in S$, and hence, $S\neq \emptyset$. Therefore, by Zorn's lemma, $S$ has a maximal element, say $m$. Now, $m=p$ if and only if $p$ is a prime $z$-element. If $m<p$, then there exists a $q\in \spec$ minimal with respect to above $m$ and $q\leqslant p$. By Lemma \ref{epzi}(\ref{mpz}), $q\in \z$, and hence, $q\neq p$. This implies either $q=m$, hence $m\in \spec$, or $m< q$,  contradicting maximality of $m$.
\end{proof}

In Lemma \ref{epzi}(\ref{ijam}), we established that $z$-elements are closed under arbitrary meets. We now seek to determine the condition under which $z$-elements are closed under finite products. To do so, we shall first review some necessary facts and definitions. Recall that an $x\in L$ is said to be \emph{idempotent} if $x^2=x$.  From Lemma \ref{flm}(\ref{mxip}), we know that every maximal element of a  multiplicative lattice $L$ is prime, and consequently, for any $m\in\mx$ and any $a\in L$, we have
$
a\leqslant m$   if and only if $a^2\leqslant m.
$ In \cite[Lemma 2.10]{MC19}, it is shown that an $a\in L$ is a $z$-element if and only if  $\mathsf{m}_x\leqslant a$ for all $x\leqslant a$. 
By a \emph{basic $z$-element} of $L$ we mean an element of the form $\mathsf{m}_a$ for some $a\in L$, and we say $L$ is \emph{semi-z-idempotent} (in short, $szi$-multiplicative lattice) if
every basic $z$-element of $L$ is idempotent. Note that every frame $L$\footnote{That is, when $\cdot$ coincides with $\wedge$ in $L$.} is an example of a $szi$-multiplicative lattice.

\begin{theorem}\label{prdz}
The product of two $z$-elements of a multiplicative lattice $L$ is a $z$-element if and only if  $L$ is a $szi$-multiplicative lattice. 
\end{theorem}

\begin{proof}
Suppose every basic $z$-element of $L$ is idempotent. Let $a$, $b\in \z$. Consider any $x\leqslant ab$. Then by Lemma \ref{bip}(\ref{pxyx}), $x\leqslant a$ and $x\leqslant b$. Hence by \cite[Lemma 2.10]{MC19}, $\mathsf{m}_x\leqslant a$ and $\mathsf{m}_x\leqslant b$. Multiplying, and using the idempotency hypothesis, and Lemma \ref{bip}(\ref{monj}), we obtain 
\[
\mathsf{m}_x=\mathsf{m}_x\mathsf{m}_x\leqslant ab,
\]
and this proves that $ab\in \z$.
	
Conversely, suppose  the product of any two $z$-elements is a $z$-element. Let $a\in L$. We aim to show that $(\mathsf{m}_a)^2=\mathsf{m}_a$. Since by Lemma \ref{flm}(\ref{mxip}), maximal elements are prime, we have the equality $\mathsf{m}_{a^2}=\mathsf{m}_a$. Since $a\leqslant \mathsf{m}_a$, once again by Lemma \ref{bip}(\ref{monj}), the inequality $a^2\leqslant (\mathsf{m}_a)^2$ holds. Since by (\ref{ijam}) and (\ref{meze}) of Lemma \ref{epzi}, $\mathsf{m}_a\in \z$, by hypothesis, $(\mathsf{m}_a)^2\in \z$, and since it is above $a^2$, we have $\mathsf{m}_{a^2}\leqslant (\mathsf{m}_a)^2$. In all, then, we have the following inequalities
\[
\mathsf{m}_a=\mathsf{m}_{a^2}\leqslant (\mathsf{m}_a)^2\leqslant \mathsf{m}_a,
\]
which yields $(\mathsf{m}_a)^2=\mathsf{m}_a$, as desired.
\end{proof}

\begin{corollary}
The $z$-elements of a $szi$-multiplicative lattice are closed under finite products.
\end{corollary}

It is natural to ask whether $z$-elements are preserved under the inverse image of a multiplicative lattice homomorphism, and we shall now derive a necessary and sufficient condition for that. This result extends \cite[Lemma 1.7]{Mas73}. Recall that a map $\phi\colon L \to L'$ is said to be a \emph{multiplicative lattice homomorphism} if $\phi$ preserves $\leqslant$, binary joins, and binary meets. 

\begin{proposition}
Let $\phi\colon L \to L'$ be a multiplicative lattice homomorphism and $j\in \mathcal{Z}(L').$ Then every element of $\phi\inv(j)$ is a $z$-element if and only if every element of $\phi\inv(m)$ is a $z$-element, whenever $m\in \mathcal{M}(L')$.
\end{proposition}

\begin{proof}
Suppose $x\in\phi\inv(j)$ and $x\in \z$, whenever $j\in \mathcal{Z}(L').$ Since by Lemma \ref{epzi}(\ref{meze}), every maximal element is a $z$-element, we immediately have $y\in \z$, for all $y\in \phi\inv(m)$, whenever $m\in \mathcal{M}(L')$. For the converse,  
suppose $j\in \mathcal{Z}(L')$. Let $\mathcal{M}^L_a=\mathcal{M}^L_b$ and $a\leqslant x$, $x\in \phi\inv (j)$, for some $a$, $b\in L$. We claim $\mathcal{M}^{L'}_{\phi(a)}=\mathcal{M}^{L'}_{\phi(b)}$. Indeed:
\[m\in \mathcal{M}^{L'}_{\phi(a)}\Leftrightarrow \phi(a)\leqslant m \Leftrightarrow a\leqslant \phi\inv(m)\Leftrightarrow b\leqslant \phi\inv(m)\Leftrightarrow \phi(b)\leqslant m\Leftrightarrow m\in \mathcal{M}^{L'}_{\phi(b)},\]
where the third bi-implication follows from \cite[Lemma 2.10]{MC19}.
Since $a\leqslant x$ implies  $\phi(a)\leqslant \phi(x)= j$, and since $j\in \mathcal{Z}(L'),$ we must have $\phi(b)\leqslant j$. Hence, $b\leqslant x$ for all $x\in \phi\inv(j)$. This proves the result. 
\end{proof}

The following proposition provides a sufficient condition for every element of a kernel to be a $z$-element.	
\begin{proposition}
$ \mathcal{K}(\phi)\subseteq  \z$  if there exists $\{a_i\}_{i\in I}\subseteq L'$ such that $\bigwedge_{i\in I} a_i=0$ and for all $i\in I$ we have $\phi\inv(a_i)\subseteq \z$.
\end{proposition}

\begin{proof}
We observe that
\[ \mathcal{K}(\phi)=\phi\inv(0)=\phi\inv\left( \bigwedge_{i\in I} a_i \right)\subseteq  \phi\inv\left(  a_i \right)\subseteq \z.\qedhere\]
\end{proof}

Our final result in this section is on $z$-elements of regular multiplicative lattices. To state the result, we need  a few definitions. An element $x$ of a multiplicative lattice $L$ is called \emph{complemented} if there exists another element $y\in L$ such that $x\wedge y=0$ and $x\vee y=1$. We say $L$ is \emph{regular} if every compact element of $L$ is a complemented element. We say an $x \in \z$ is \emph{strong}\footnote{The terminology is adapted from the corresponding notion of strong $z$-ideals (see \cite[p.\,281]{Mas73}).} if $x$ is the meet of maximal elements in $L$.

\begin{theorem}\label{regz}
Every element in a regular multiplicative lattice is a strong $z$-element.
\end{theorem}

\begin{proof}
Follows from (\ref{ijam}), (\ref{meze}) of Lemma  \ref{epzi}, and \cite[Theorem 5.2]{AAJ95}.
\end{proof}

\section{$z$-closure operators}\label{zco}

The study of a class of ideals using a closure operator is not new. For example, the concept of $k$-ideals in semirings have been defined and studied using a closure operator (see \cite{SA92}). In this section, our focus is to systematically study the properties of `$z$-closure operators' and utilize them to derive additional results on $z$-elements. Furthermore, in the next section, we will use these $z$-closure operators to establish correlations between newly introduced subclasses of $z$-elements and their known counterparts. What we refer to as the $z$-closure operator in Definition \ref{clkdef} was originally introduced as $i_z$ in \cite[Lemma 2.7]{MC19}, and the same is well-known (denoted by $I_z$) in the context of rings.

Let us suppose $L$ is a multiplicative lattice. The \emph{$z$-closure} operator $\cz\,\colon L \to L$  is defined by
\begin{equation}
\label{clkdef}
\cz(a)=\bigwedge\left\{z \in \z \mid a\leqslant z \right\} \qquad (a\in L).
\end{equation} 	
\begin{remark}
A $z$-closure operator can also be defined as $\cz(a)=\mathcal{I}\mathcal{V}(a),$ where $\mathcal{V}(a)=\left\{ z \in \z \mid a \leqslant z\right\}$ and $\mathcal{I}(S)=\bigwedge S$ $(S\subseteq \z)$.  
\end{remark}

In the next lemma, we shall gather some  properties of $z$-closure operators. Some of  these results involve the notion of radical of an element. For that, recall from \cite{AAJ95} that the \emph{radical} of an element $x$ of $L$ is defined by
\[
\sqrt{x}=\bigvee\left\{y\in L\mid \;y\;\text{is compact},\;y^n\leqslant x,\;\text{for some}\; n\in \mathds{N}^+\right\}=\bigwedge\left\{p\in \spec\mid x\leqslant p\right\}.	
\]
An element $x\in L$ is called a \emph{radical element} if $x=\sqrt{x}$. With our  assumption of every finite product of compact elements of $L$ is a compact
element, we further have (see \cite[p.\,2]{AAJ95}):
\begin{equation}\label{rmpp}
\sqrt{x}=\bigwedge\left\{p\in \spec\mid p\;\text{is minimal prime over}\; x\right\}.	
\end{equation}

\begin{proposition}\label{lclk}
For any elements  $a$,  $b$ of a multiplicative lattice $L$, the following hold.
\begin{enumerate}
		
\item\label{iclk} $\cz(a)$ is the smallest $z$-element such that $a\leqslant \cz(a)$.\!\!\footnote{This property along with (\ref{clcl}) justifies the terminology  \emph{closure operator} for $\cz$ (see \cite{NR88}).}

\item\label{altd}   $\cz(a)=a$ if and only if $a\in \z$.\!\!\footnote{This result gives, yet, another alternative definition of a $z$-element.}

\item \label{ckr}
$ \cz(a) =1$ if and only if $a=1$.
		
\item \label{ssz}
$\cz(0)=0$, whenever $L$ is semisimple.
			
\item\label{ijcl} If $a\leqslant b$, then $\cz(a)\leqslant \cz(b).$

\item\label{clcl} $\cz(\cz(a))=\cz(a).$
		
\item \label{clsq}  $\sqrt{a}\leqslant \cz(a)$.

\item\label{craca} $\sqrt{\cz(a)}=\cz(\sqrt{a})$.

\item\label{clijk}  $\cz(ab)=\cz(a)\cz(b)$, whenever $L$ is a $szi$-multiplicative lattice.

\item\label{tccl}  $\cz(ab)=\cz(a\wedge b)=\cz(a)\wedge \cz(b)$.

\item\label{clab} $\cz(a) \vee \cz(b)\leqslant \cz( a\vee b )=\cz(\cz(a)\vee\cz(b)).$

%\item\label{caicb} $\cz(ab)\leqslant \cz(a)\wedge \cz(b)$.

\item\label{clvm} $\cz(a)\leqslant \mathsf{m}_a$. If $a\in \z$, then $\cz(a)= \mathsf{m}_a$.

\item\label{clan} $\cz(a^n)=\cz(a)$, for any $n\in \mathds{N}^+$.
\end{enumerate}
\end{proposition}

\begin{proof}
(1). \cite[Lemma 2.7]{MC19}.
%From Definition \ref{clkdef}, it is clear that $a \leqslant \cz(a)$ and from Lemma  \ref{epzi}(\ref{ijam}) it follows that $\cz(a)\in \z$. If $z'$ is a $z$-element of $L$ such that $a\leqslant z'$, then obviously $ \bigwedge\left\{z \in \z\mid a\leqslant z\right\}\leqslant z',$ and this proves the claim. 

(2). If $\cz(a)=a$, then $a\in \z$  by (\ref{iclk}). Conversely, if $a$ is a $z$-element, then by Definition \ref{clkdef}, we have the desired identity.

(3). Straightforward.

(4). To have the claim, it suffices to show that $0\in \z$, and that follows from Lemma \ref{epzi}(\ref{ziz}).
%Let $\mathcal{M}_a=\mathcal{M}_b$ and $a\leqslant 0$. Then $a=0$, and hence, $\mathcal{M}_b=\mathcal{M}_0=\mx$, implying  \[b\leqslant \bigwedge\left\{m\mid m\in \mx \right\}=\jr=0.\] Therefore, $b=0$. This implies that $0\in \z$.

(5)--(6). \cite[Lemma 2.8(1--2)]{MC19}.
%From 1, it follows that $\cz(a)\leqslant \cz(\cz(a)).$ Conversely, $\cz(\cz(a))\leqslant z$ for all $z\in \z$ such that $\cz(a)\leqslant z$. But by 1, we have $a\leqslant \cz(a)$. Therefore, $\cz(\cz(a))\leqslant \cz(a).$

(7). Suppose $x \in \z$  such that $a \leqslant x$. We aim to show that $\sqrt{a} \leqslant x$.
Let us
consider any $y \in L$ such that $y$ is compact and $y^n \leqslant a$ for some $n \in \mathds{N}^+$. Then we have $y\leqslant \sqrt{a}$. Since $y^n \leqslant a \leqslant x$ and $x\in \z$, we have $y \leqslant x$.
Therefore, $\sqrt{a} \leqslant x$ for any  $x\in \z$ such that $a \leqslant x$. This implies that $\sqrt{a}\leqslant \cz(a)$.

(8). 
By (\ref{clsq}), we have $\sqrt{a}\leqslant \cz(a)$. Applying (\ref{ijcl}), (\ref{clcl}), and (\ref{iclk}) respectively, we get \[\cz(\sqrt{a})\leqslant \cz(\cz(a))=\cz(a)\leqslant \sqrt{\cz(a)}.\] For the converse inequality, by applications of (\ref{clsq}), (\ref{clcl}), and (\ref{ijcl}) respectively, we have \[\sqrt{\cz(a)}\leqslant \cz(\cz(a))=\cz(a)\leqslant \cz(\sqrt{a}).\]

(9). By (\ref{iclk}), we have $a\leqslant \cz(a)$ and $b\leqslant \cz(b)$, and hence $ab\leqslant \cz(a)\cz(b)$ by Lemma \ref{bip}(\ref{monj}). However, $\cz(a)\cz(b)\in \z$,  by Theorem \ref{prdz}. Therefore, it is sufficient to show that $\cz(a)\cz(b)$ is the smallest $z$-element that is above $ab$. Suppose $c\in \z$ such that $ab\leqslant c$. Let $\mathrm{Min}(c)$ be the set of all minimal prime elements of $L$ that are above $c$. If $p\in \mathrm{Min}(c)$, then by Lemma \ref{epzi}(\ref{mpz}), $p\in \z$, and hence by Definition \ref{rmpp} and Lemma \ref{epzi}(\ref{ijam}), we have \[\sqrt{c}=\bigwedge\left\{p\mid p\in \mathrm{Min}(c)\right\}\in \z.\] Therefore, by applying (\ref{clsq}) and (\ref{craca}), we get
\[\sqrt{c}=\cz(\sqrt{c})=\sqrt{\cz(c)}\leqslant \cz(\cz(c))=\cz(c)=c.\]
Since for each $p\in \mathrm{Min}(c)$ has the property: $ab\leqslant p$ implies $a\leqslant p$ or $b\leqslant p$, from this, we obtain
$\cz(a)\cz(b)\leqslant p^2\leqslant p,$ and therefore
\[\cz(a)\cz(b)\leqslant\bigwedge\left\{p\mid p\in \mathrm{Min}(c)\right\}=c,\]
implying that $\cz(a)\cz(b)$ is the smallest $z$-element of $L$ that is above $ab.$

(10). Although the proof of the second identity follows from \cite[Lemma 3.6]{MC19}, however, following \cite[Lemma 1.3(c)]{AAT13}, here we show $\cz(ab)=\cz(a)\wedge \cz(b)$, and from that we have the second identity for free. Since by Lemma \ref{bip}(\ref{pxyx}), $ab\leqslant a$, $ab\leqslant b$, applying (\ref{ijcl}), (\ref{iclk}), Lemma \ref{epzi}(\ref{ijam}), and Lemma \ref{bip}(\ref{mul}), we have \[\cz(ab)\leqslant \cz(a\wedge b)\leqslant \cz(a)\wedge \cz(b)\in \z.\]
Therefore, to obtain the desired identities, it is sufficient to show that $\cz(a)\wedge \cz(b)$ is the smallest $z$-element that is above $ab$. The rest of  the argument is similar to (\ref{clijk}).

(11). The first inequality follows from 5. By (\ref{iclk}), we have $a\leqslant \cz(a)$ and $b\leqslant \cz(b)$. This  yields $a\vee b \leqslant \cz(a)\vee \cz(b)$. Now by applying (\ref{ijcl}), we obtain   \[\cz(a\vee b)\leqslant \cz(\cz(a)\vee\cz(b)).\] Since $a,$ $b\leqslant a\vee b$, by (\ref{iclk}), we get $a, $ $b\leqslant \cz(a\vee b)$. From here, by  applications of (\ref{clcl}) and (\ref{ijcl}) give us $\cz(a)\vee \cz(b)\leqslant \cz(a\vee b).$ Once again, applying (\ref{clcl}) and (\ref{ijcl}) yields the desired inequality.

%10. To obtain the first identity, by Lemma \ref{bip}(\ref{pxyx}), we have $ab\leqslant a$ and $ab\leqslant b$. Applying \ref{ijcl} on these yields $\cz(ab)\leqslant \cz(a)\wedge \cz(b).$

(12). Since every maximal element is a $z$-element, the desired inequality follows from the definitions of $\cz(a)$ and $\mathsf{m}_a$. If $a\in \z$, then by (\ref{altd}) and Lemma \ref{epzi}(\ref{alze}), we have $\cz(a)=a=\mathsf{m}_a.$

(13). 
Since $a^n\leqslant a$, we have $\cz(a^n)\leqslant \cz(a).$ For the converse, suppose $k$ is a compact element such that $k\leqslant \cz(a)$. This implies that $k\leqslant x$ for all $x\in \z$ such that $a\leqslant x$. Since $x\in \z$, we have $a^n\leqslant x$, and hence $k\leqslant \cz(a^n)$, proving that $\cz(a)\leqslant\cz(a^n)$.
\end{proof}

\begin{remark}
In Proposition \ref{lclk}, the results (\ref{iclk}) and (\ref{clan}) extend the corresponding $z$-ideals versions (see \cite[p.\,281]{Mas73}); (\ref{altd}) extends   \cite[remark on p. 347]{Ben21}; (\ref{ijcl}) extends \cite[Lemma 3.6(1)]{MJ20}; (\ref{tccl}) extends \cite[Lemma 1.3(c)]{AAT13}; (\ref{clab}) extends    \cite[Lemma 2.2(1)]{Ben21}; (\ref{clsq}) extends \cite[Proposition 3.7]{MJ20};  (\ref{craca}) extends \cite[Proposition 3.9(3)]{MJ20}.
\end{remark}

We have discussed about closedness of $z$-elements under meets and products. The following proposition gives some equivalent criteria on closedness of $z$-closure operators under joins. This result extends \cite[Proposition 3.1(b)]{Mas80}.

\begin{proposition}\label{cujo}
In a multiplicative lattice $L$, the following are equivalent.
\begin{enumerate}
\item\label{cuj} If $a$, $b\in \z$, then $a\vee b \in \z$.

\item If $a$, $b\in L$, then $\cz(a\vee b)=\cz(a)\vee \cz(b).$

\item If $\{a_i\}_{i\in I}\subseteq \z$, then $\bigvee_{i\in I} a_i\in \z.$

\item If $\{a_i\}_{i\in I}\subseteq L$, then $\cz\left(  \bigvee_{i\in I} a_i\right)=\bigvee_{i\in I} \cz(a_i)$.
\end{enumerate}
\end{proposition}

\begin{proof}
The proofs of (1)$\Leftrightarrow$(2), (3)$\Leftrightarrow$(4), and (3)$\Rightarrow$(1) are straightforward. Therefore, what remains is to show (1)$\Rightarrow$(3). Suppose  $\mathcal{M}(a)\supseteq \mathcal{M}(b)$ and $b\leqslant \bigvee_{i\in I} a_i$ for some $a,$ $b\in L$. Since $L$ is compactly generated, there exists a finite subset $J$ of $I$ such that $b\leqslant \bigvee_{j\in J} a_j$. Since $\bigvee_{j\in J} a_j\in \z$ by (\ref{cuj}), we must have \[a\leqslant \bigvee_{j\in J} a_j\leqslant \bigvee_{i\in I} a_i,\] implying that $\bigvee_{i\in I} a_i\in \z.$
\end{proof}

Our next result is on a characterization of $pz$-multiplicative lattices and it extends \cite[Proposition 2.11]{AM22}. For the statement of the proposition, we need to recall that a proper element $q$ of $L$ is said to be \emph{semiprime} if for all $a\in L$ with $a^2\leqslant q$ implies that $a\leqslant q$.

\begin{proposition}\label{cpzl}
For a multiplicative lattice $L$, the following statements are equivalent.
\begin{enumerate}
\item $L$ is a $pz$-multiplicative lattice.

\item Every semiprime element of $L$ is a $z$-element.

\item\label{maiv} $\mathsf{m}_a=\mathcal{I}\mathcal{V}_{\mathcal{P}}(a),$ for all $a\in L$, where $\mathcal{V}_{\mathcal{P}}(a)=\left\{p\in \spec\mid a\leqslant p\right\}.$\!\!\!\footnote{Each $\mathcal{V}_{\mathcal{P}}(a)$ is nothing but a closed set for a Zariski topology on $\spec$.}

\item $\cz(a)=\sqrt{a},$ for all $a\in L$.
\end{enumerate}
\end{proposition}

\begin{proof}
(1)$\Rightarrow$(2): Obvious. 

(2)$\Rightarrow$(3): For all $a\in L$, it is clear that $\sqrt{a}=\mathcal{I}\mathcal{V}_{\mathcal{P}}(a)\leqslant\mathsf{m}_a$. If $p\in \mathcal{V}_{\mathcal{P}}$, then $p\in \z$, and hence, by \cite[Lemma 2.10]{MC19}, we have $\mathsf{m}_a\leqslant p,$ which implies that $\mathsf{m}_a\leqslant\mathcal{I}\mathcal{V}_{\mathcal{P}}(a).$ 

(3)$\Rightarrow$(4): From Proposition \ref{lclk}(\ref{clsq}), we have $a\leqslant \sqrt{a}\leqslant \cz(a).$ Therefore, it is sufficient to show that $\sqrt{a}\in \z$. By (\ref{maiv}) and from the definition of $\sqrt{a}$, we have
\[\mathsf{m}_a=\mathcal{I}\mathcal{V}_{\mathcal{P}}(a)=\sqrt{a}, \]
and $\mathsf{m}_a\in \z$ by (\ref{ijam}) and (\ref{meze}) of Lemma \ref{epzi}.

(4)$\Rightarrow$(1): Since $\sqrt{p}=p$, for all $p\in \spec$, we have the desired implication immediately.
\end{proof}

\subsection{Nuclei}\label{sonu}

In this short subsection we shall view $z$-closure operators as  nuclei. Let us suppose $L$ is a multiplicative lattice. When $L$ is a $szi$-multiplicative lattice, every $z$-closure operator $\cz$ is a \emph{quantic nucleus} in the sense of \cite[p.\,218]{NR88}. This is because $\cz$ satisfies properties (\ref{iclk}), (\ref{clcl}), and (\ref{clijk}) from Proposition \ref{lclk}.
When $L$ becomes a frame, replacing property (\ref{clijk})  by second identity of (\ref{tccl}) implies that every $\cz$ is a \emph{nucleus} in the sense of \cite[Definition 1]{Sim78}. Since properties (\ref{ckr}) and (\ref{tccl}) from Proposition \ref{lclk} hold for every $\cz$ defined on any multiplicative lattice, every $\cz$ is a \emph{multiplicative nucleus} in the sense of \cite[Definition 1.1]{BH85}.

Here we gather some results on $z$-closure operators when viewed as nuclei of various types. The proofs of these results are either identical or similar to those presented in \cite{NR88, BH85} in more general context.

\begin{proposition}
If $L$ is a $szi$-multiplicative lattice, then $\cz$ (as a quantic nucleus) satisfies the following identities:
\[\cz(ab)=\cz(a\cz(b))=\cz(\cz(a)b)=\cz(\cz(a)\cz(b)),\]
for all $a,$ $b\in L$.
\end{proposition}
Defining the join in $\z$  by $\bigvee'_{i\in I} a_i=\cz\left( \bigvee_{i\in I} a_i\right),$ we obtain the following (\textit{cf}. \cite[Theorem 2.1]{NR88}).

\begin{theorem}
If $\cz$ is a quantic nucleus, then $\z$ is a multiplicative lattice via $a\odot b =  \cz(ab);$
and $\cz\colon L \to \z$ is a multiplicative lattice homomorphism. Moreover, every surjective multiplicative lattice
homomorphism arises (up to
isomorphism) in this manner.
\end{theorem}

Since $\cz$ satisfies the property: $\cz(ab)=\cz(a)\wedge \cz(b)$ (see Proposition \ref{lclk}(\ref{tccl})), it follows that $\z$ is a frame with $\odot = \wedge$, and $\cz$ is called a \emph{localic nucleus} (see \cite[p.\,219]{NR88}, \cite{Ban94}).
Since $\cz$ is a multiplicative nucleus, in fact, we get more  (\textit{cf}. \cite[Lemma 1.2]{BH85}).

\begin{theorem}
$\z$ is a compact frame.
\end{theorem}

\section{Some classes of $z$-$\star$\,elements}
\label{sczi}

The aim of this section is to introduce some distinguished classes of $z$-$\star$\,elements (as discussed in \textsection \ref{intro}) of multiplicative lattices and study some of their properties. These are subclasses of $z$-elements of a multiplicative lattice. We shall start with irreducible-type elements.

Let us suppose that $L$ is a multiplicative lattice. We say that an $s\in L$ is \emph{ strongly irreducible}\footnote{In \cite{AAJ95}, what we have referred to as strongly irreducible and irreducible elements are termed as \emph{meet prime} and \emph{meet-irreducible}, respectively. We have chosen our terminology to be consistent with the corresponding types of ideals in rings.} (resp. irreducible) if for all $a,$ $b\in L$ with $a\wedge b \leqslant s$ (resp. $a\wedge b=s$) implies that $a\leqslant s$ (resp. $a=s$) or $b\leqslant s$ (resp. $b=s$).   
If $s$, $a$, and $b\in \z$, then we say  $s$ is  \emph{$z$-strongly irreducible} (resp. \emph{$z$-irreducible}).  By $\irs(L)$ (resp. $\irr(L)$), we  denote the set of $z$-strongly irreducible (resp. $z$-irreducible) elements of $L$.

\begin{lemma}
In a multiplicative lattice, every prime element is  strongly irreducible and every $z$-strongly irreducible element is $z$-irreducible.
\end{lemma}  

\begin{proof}
The first claim follows from Lemma \ref{bip}(\ref{mul}). To show the second assertion, suppose $s\in \irs$ and $a\wedge b= s$, for some $a,$ $b\in L$. Since $s\in \irs$, either $a\leqslant s$ or $b\leqslant s$. Considering the first case, we have
\[s=a\wedge b\leqslant a\leqslant s,\]
and hence, $a=s$. If $b\leqslant s$, then similarly it follows that $b=s$. This proves that $s\in \irr.$
\end{proof}

We anticipate that the equivalence formulation (\ref{kxkx}) holds for both $z$-irreducible and $z$-strongly irreducible elements. We will now prove this in the next proposition.

\begin{proposition}
\label{eqsi}
An element $s$ of a multiplicative lattice $L$ is $z$-strongly irreducible ($z$-irreducible)  if and only if $s$ is  strongly irreducible (irreducible) element as well as a $z$-element of $L$. 
\end{proposition}

\begin{proof}
We give a proof for $z$-strongly irreducible elements, that for $z$-irreducible elements will be similar.	Let $s\in \irs(L)$  and $b$, $b'\in L$  such that $b\wedge b'\leqslant s.$ Applying Proposition \ref{lclk}(\ref{tccl}) and  Proposition \ref{lclk}(\ref{ijcl}), yields
\[\cz(b)\wedge \cz(b')=\cz(b\wedge b')\leqslant \cz(s)=s.\]
By assumption, this implies that  $ \cz(b)\leqslant s$ or $ \cz(b')\leqslant s$. By Proposition \ref{lclk}(\ref{iclk}), we now have the claim. The proof of the converse is obvious.
\end{proof}

For rings, it is  well-known  that every proper ideal is contained in a minimal strongly irreducible  ideal (see  \cite[Theorem 2.1(ii)]{Azi08}). We shall see now that this result can be extended to $z$-strongly irreducible elements.

\begin{proposition}
%Azizi Th 2.1
Every proper $z$-element of  a multiplicative lattice $L$ is below a minimal $z$-strongly irreducible element.
\end{proposition}

\begin{proof}
Let us suppose $a\in \z$  such that $a\neq 1$. We consider the chain $(\mathcal{E}, \leqslant)$,  where \[\mathcal{E}=\left\{b\mid a\leqslant b,\;b\in \irs(L)\right\}.\]
By Lemma \ref{flm}(\ref{exml}), every proper $z$-element is below a maximal element. Since every maximal element $m$ of $L$ is prime (see Lemma \ref{flm}(\ref{mxip})), and hence, is  strongly irreducible (follows Lemma \ref{bip}(\ref{mul})), by  Proposition \ref{eqsi}, it follows that $m\in \irs(L)$. Therefore, $\mathcal{E}\neq \emptyset$. By applying Zorn's lemma, we conclude that $\mathcal{E}$ has a minimal element, which is our desired minimal $z$-strongly irreducible element.
\end{proof}

The following result is going to tell us what happens to $\z$, when all $z$-elements of  $L$ are $z$-strongly irreducible. This result extends  \cite[Lemma 3.5]{Azi08}.

\begin{proposition}
	%Azizi L.3.5
If $L$ is a multiplicative lattice in which every $z$-element is $z$-strongly irreducible, then $\z$ is totally ordered
\end{proposition}

\begin{proof}
Suppose for any $a$, $b\in \z$ are $z$-strongly irreducible. Since $a\wedge b \leqslant a\wedge b$, we have either $a\leqslant a\wedge b \leqslant b$ or $b\leqslant a\wedge b\leqslant a$. This implies that $\z$ is  totally ordered. 
\end{proof}

Next, we aim to present a result on the representation of $z$-elements in terms of $z$-irreducible elements in a special type of multiplicative lattices. We say that a multiplicative lattice $L$ is  \emph{$z$-Noetherian} if every ascending chain of $z$-elements in $L$ is eventually  stationary. 	
The following proposition provides an extension of  \cite[Proposition 7.3]{Nas18}. 

\begin{theorem}\label{nssi}
%Nasehpour2018, Proposition 7.3
Let $L$ be a $z$-Noetherian multiplicative lattice. Then every $z$-element of $L$ can be represented as a meet of a finite number of $z$-irreducible elements of $L$. 
\end{theorem}

\begin{proof}
Let us suppose
\[\mathcal{F}=\left\{b\in  \z \mid b\neq \bigwedge_{i=1}^n q_i,\;q_i \in \irr(L)\right\}.\] It suffices to show that $\mathcal{F}=\emptyset$.
Since $L$ is $z$-Noetherian, $\mathcal{F}$ must have a maximal element, denoted by $a$. Since $a \in \mathcal{F}$, $a$ is not a finite
meet of $z$-irreducible elements of $L$. This implies that $a$ is not $z$-irreducible. Hence, there are $z$-elements $b$ and $b'$ such that  $a<b$, $a<b'$, and $a = b \wedge b'.$ Since $a$ is
a maximal element of $\mathcal{F}$, we must have $b,$ $b' \notin \mathcal{F}.$ Therefore, $b$ and $b'$ are a finite meet of
$z$-irreducible elements of $L$. This, in turn, implies that $a$ is also a finite meet of $z$-irreducible
elements of $L$, a contradiction.
\end{proof}

We shall now introduce prime-type $z$-elements. A proper $z$-element $p$ of a multiplicative lattice $L$ is called \emph{$z$-prime}  if for all $a,$ $b\in \z$ with $ab\leqslant p$ implies that $a\leqslant p$ or $b\leqslant p$. We shall denote by $\specz$ the set of all $z$-prime elements of $L$. From Lemma \ref{bip}(\ref{mul}), it follows that every $z$-prime element is $z$-strongly irreducible. 
Under an additional assumption on $L$,  we now show that the equivalence formulation (\ref{kxkx}) holds for the $z$-prime elements. 

\begin{proposition}
\label{eqp} 
Suppose $L$ is a $szi$-multiplicative lattice. An element $p$ of $L$ is $z$-prime if and only if $p\in \z$ and $p\in \spec$. 
\end{proposition}

\begin{proof}
It is evident that if $p\in \z$  as well as $p\in \spec$, then  $p$ is a $z$-prime element.
For the converse, suppose $p\in \specz$  and $xy\leqslant p$, for some $x$, $y\in L$.  Then from Lemma \ref{lclk}(\ref{clijk}), it follows that
\[\cz(x) \cz(y) \leqslant \cz(xy) \leqslant \cz(p))=p.\]
Since $p\in \specz$, we  have  $x\leqslant  \cz(x)\leqslant p$ or $y \leqslant\cz(y)\leqslant p.$ This implies $p\in \spec$.
\end{proof}
\begin{proposition}
\label{kmkp}
Every non-trivial multiplicative lattice $L$ contains a minimal $z$-prime element.
\end{proposition}

\begin{proof}
The existence of minimal prime elements is guaranteed by the supposition of that finite products of compact elements are compact (see \cite{AAJ95}).
Suppose   $L$ is a non-trivial multiplicative lattice. Then by  Lemma \ref{flm}(\ref{exml}), it follows that $L$ has a maximal element, say $m$, which by Lemma \ref{flm}(\ref{mxip}) and Lemma  \ref{epzi}(\ref{meze}) is a prime element as well as a $z$-element. Hence, by definition, $m$ is a $z$-prime element, implying that the set $\specz$ is nonempty. The claim now follows from a routine application of Zorn's lemma.	
\end{proof}

For  $z$-Noetherian multiplicative lattices, the following proposition provides information regarding the cardinality of minimal $z$-prime elements. This result extends \cite[Proposition 6.7]{Les15}.

\begin{proposition}
\label{wnknp} %\cite[Proposition 6.7]{Lescot15}
If $L$ is a $z$-Noetherian multiplicative lattice, then the set of minimal $z$-prime elements of  $L$ is finite.
\end{proposition}

\begin{proof}
If $L$ is $z$-Noetherian, then the topological space $\specz$ (endowed with Zariski topology)  is also Noetherian, and thus $\specz$ has  finitely many irreducible
components. Now every irreducible closed subset of $\specz$
is of the form \[\mathcal{V}_z(p)=\left\{q\in \specz \mid p\leqslant q\right\},\] where $p$ is a minimal $z$-prime element. Thus $\mathcal{V}_z(p)$ is irreducible component if and only if $p$ is a minimal $z$-prime element. Hence, the number of minimal $z$-prime elements of $L$ is finite.
\end{proof}

We now focus in introducing another class of prime-type elements of multiplicative lattices. A proper $z$-element $p$ of a multiplicative lattice $L$ is called \emph{$z$-semiprime} if for all $a\in \z$ with $a^2\leqslant p$ implies $a\leqslant p$. 
Similarly to $z$-prime elements, we also establish an equivalent formulation (\ref{kxkx}) for $z$-semiprime elements under the same additional assumption.

\begin{proposition}\label{eqpp}
An element $q$ of a $szi$-multiplicative lattice  $L$ is $z$-semiprime if and only if $q$ is a $z$-element and also a semiprime element of $L$. 
\end{proposition}

\begin{proof}
If $q\in \z$  and also a semiprime element of  $L$, then obviously  $q$ is $z$-semiprime. For the converse, let us suppose that $q$ is a $z$-semiprime element and $a^2\leqslant q$ for some $a\in L.$ Then
\[(\cz(a))(\cz(a))\leqslant\cz(a^2)\leqslant \cz(q)=q,\]
where the first and second inclusions respectively follow from Lemma \ref{lclk}(\ref{clijk}) and Lemma \ref{lclk}(\ref{ijcl}), whereas the last equality follows from the fact that $q\in \z$ and Lemma \ref{lclk}(\ref{altd}). Since $q$ is $z$-semiprime, $\cz(a)\leqslant q$, and by  Lemma \ref{lclk}(\ref{iclk}), we have $a\leqslant q.$ 
\end{proof}

We shall now see a relation between prime-type and irreducible-type $z$-elements, and the following results extends \cite[Proposition 7.36]{Gol99}.

\begin{proposition}\label{vpss}
	%Golan 7.36
A $z$-element of a $szi$-multiplicative lattice $L$ is $z$-prime if and only if it is $z$-semiprime and $z$-strongly irreducible. 
\end{proposition} 

\begin{proof}
Let $p\in \specz$. Then by Proposition \ref{eqp}, $p\in \z$  and $p\in \spec$. This implies that $p$ is $z$-semiprime by Proposition \ref{eqpp}. From Lemma \ref{bip}(\ref{mul}) and Proposition \ref{eqsi}, it follows that $p\in \irs(L)$. Conversely, let $p$ be a $z$-semiprime as well as a $z$-strongly irreducible element. Let  $a$, $b\in \z $ satisfying $ab\leqslant p$. Then 
\[(a\wedge b)^2\leqslant a b\leqslant p.\]
Since $p$ is $z$-semiprime, this implies $a\wedge b\leqslant p$. But $p$ is also $z$-strongly irreducible, and so $a\leqslant p$ or $b\leqslant p.$
\end{proof}

To introduce our final class of $z$-$\star$ element, we recall from \cite{Dil62} that an element $r$ of a multiplicative lattice $L$ is said to be \emph{primary} if for $x$, $y\in L$ with $xy\leqslant r$ implies $x\leqslant r$ or $y^n\leqslant r$ for some $n\in \mathds{N}^+$. If in this definition, $r$, $x$, $y\in \z$, then we say $r$ is a \emph{$z$-primary} element of $L$. Like $z$-prime elements, the equivalence formulation (\ref{kxkx}) also holds for $z$-primary elements of $szi$-multiplicative lattices, and proof of that is similar to Proposition \ref{eqp}.

\begin{proposition}\label{zprim}
In a $szi$-multiplicative lattice $L$, an  $r\in L$  is $z$-primary if and only if $r\in \z$  and $r$ is a primary element.
\end{proposition}

Thanks to Proposition \ref{zprim}, many of the results (like primary decomposition) relating  primary elements in \cite{Dil62} carry over directly to $z$-primary elements.

\begin{cremark}
Here are some questions on $z$-elements of multiplicative lattices that may be worth exploring.
\begin{enumerate}
\item[$\bullet$] In Lemma \ref{epzi}(\ref{ziz}), we observed that $0\in \z$, whenever $L$ is semisimple. It would be interesting to know whether the converse is also true.

\item[$\bullet$] A sufficient condition for the identity $\cz(a) = \mathsf{m}_a$ to hold is provided in Proposition \ref{lclk}(\ref{clvm}). We are not aware of any necessary and sufficient condition for this.

\item[$\bullet$] Although Proposition \ref{cujo} provides some equivalent conditions on the closedness of $\cz$ under joins, it does not provide a characterization for multiplicative lattices to possess this property for $\cz$. This may be worth investigating.

\item[$\bullet$] In this paper, we have proved the equivalence formulation (\ref{kxkx}) for $z$-prime, $z$-semiprime, and $z$-primary elements under the assumption that $L$ is a $szi$-multiplicative lattice. Whether this assumption is necessary remains unknown to us.

\item[$\bullet$] It could be interesting to know for which other distinguished classes of elements  the equivalence formulation (\ref{kxkx}) holds, potentially with suitable assumptions on the underlying lattice.
	
\item[$\bullet$] The notion of $z$-ideals has been extended to arbitrary noncommutative rings in \cite{MJ20}. A similar extension for $z$-elements is worth it to study for arbitrary quantales. 

\item[$\bullet$] The notion of higher order $z$-ideals of (commutative) rings has been introduced and studied in \cite{DI16}. It is easy to see that the definition of a $z$-element can also be extended to higher order $z$-elements. Therefore, studying these higher-order $z$-elements of multiplicative lattices would be of interest.

\item[$\bullet$] In \cite{DG23}, it has been  discussed that the class of strongly irreducible ideals is the largest class on which a hull-kernel topology can be imposed. Considering $z$-strongly irreducible elements as the corresponding class of elements, it could be intriguing to explore the topological properties of this class and its subspaces, such as $z$-prime, $z$-semiprime, $z$-minimal prime, maximal elements, \textit{etc}.
\end{enumerate}
\end{cremark}

%\section*{Acknoledgement}

\bibliographystyle{smfalpha}

\end{document}